\documentclass[12pt,a4paper]{amsart}
\usepackage{amssymb,enumerate,amsmath,xcolor}

\usepackage{hyperref}

\newcommand{\FF}{{\mathbb{F}}}

\newcommand{\bB}{{\mathbf B}}
\newcommand{\bG}{{\mathbf G}}
\newcommand{\bO}{{\mathbf O}}
\newcommand{\bZ}{{\mathbf Z}}

\newcommand{\cE} {\mathcal{E}}

\newcommand{\fA} {\mathfrak A}
\newcommand{\fS} {\mathfrak S}

\newcommand{\IBr}{{\operatorname{IBr}}}
\newcommand{\Irr}{{\operatorname{Irr}}}
\newcommand{\SC}{{\operatorname{sc}}}
\newcommand{\St}{{\operatorname{St}}}
\newcommand{\Syl}{{\operatorname{Syl}}}

\newcommand{\SL}{{\operatorname{SL}}}
\newcommand{\PSL}{{\operatorname{PSL}}}
\newcommand{\SU}{{\operatorname{SU}}}
\newcommand{\PGU}{{\operatorname{PGU}}}
\newcommand{\PSU}{{\operatorname{PSU}}}
\newcommand{\PSp}{{\operatorname{PSp}}}

\newcommand{\tw}[1]{{}^#1\!}

\let\eps=\epsilon
\let\ga=\gamma
\let\vhi=\varphi

\def\cent#1#2{{\bf C}_{#1}(#2)}
\def\norm#1#2{{\bf N}_{#1}(#2)}

\newtheorem{thm}{Theorem}[section]
\newtheorem{lem}[thm]{Lemma}
\newtheorem{cor}[thm]{Corollary}
\newtheorem{prop}[thm]{Proposition}

\newtheorem*{thmA}{Theorem A}
\newtheorem*{property}{Property $(*)$}

\theoremstyle{definition}

\begin{document}

\title[Decomposition numbers and Sylow normalisers]{Decomposition numbers in the principal block\\ and Sylow normalisers}

\author{Gunter Malle}
\address[G. Malle]{FB Mathematik, RPTU, Postfach 3049,
  67653 Kaisers\-lautern, Germany.}
\email{malle@mathematik.uni-kl.de}

\author{Noelia Rizo}
\address[N. Rizo]{Departament de Matem\`atiques, Universitat de Val\`encia,
 46100 Burjassot, Val\`encia, Spain}
 \email{noelia.rizo@uv.es}

\begin{abstract}
If $G$ is a finite group and $p$ is a prime number, we investigate the
relationship between the $p$-modular decomposition numbers of characters of
height zero in the principal $p$-block of $G$ and the $p$-local structure of
$G$. In particular we prove that, under certain conditions on the non-abelian
composition factors of $G$, $d_{\chi 1_G}\neq 0$ for all irreducible characters
$\chi$ of degree prime to~$p$ in the principal $p$-block of $G$ if, and only if,
the normaliser of a Sylow $p$-subgroup of $G$ has a normal $p$-complement.
\end{abstract}

\thanks{The first author gratefully acknowledges financial support by the DFG --- Project-ID 286237555.
The second author is supported by Ministerio de Ciencia e Innovaci\'on (Grant PID2019-103854GB-I00 and Grant PID2022-137612NBI00 funded by MCIN/AEI/10.13039/501100011033 and “ERDF A way of making Europe”) and a CDEIGENT grant CIDEIG/2022/29 funded by Generalitat Valenciana.}



\dedicatory{To the memory of Gary Seitz}

\date{\today}

\maketitle


\section{Introduction}

Let $G$ be a finite group, $p$ a prime number, $\Irr(G)$ the set of irreducible
ordinary characters of $G$ and $\IBr(G)$ the set of irreducible $p$-Brauer
characters of $G$. Then the restriction $\chi^\circ$ of any $\chi\in\Irr(G)$ to
the set of $p$-regular elements of $G$ can be written as
$$\chi^\circ=\sum_{\vhi\in\IBr(G)}d_{\chi\vhi}\vhi,$$
where the $d_{\chi\vhi}$ are uniquely determined non-negative integers. These are
called the $p$-modular \emph{decomposition numbers of $G$}, and a great deal of
literature is devoted to understand them.
\medskip

In \cite{NT,NT2} Navarro and Tiep initiated the investigation on relations
between $p$-decomposition numbers and properties of Sylow $p$-normalisers,
considering two different settings. In \cite{NT} they conjectured that if $p>3$
then $d_{\chi 1_G}\neq 0$ for all $\chi\in\Irr_{p'}(G)$, that is, for all
irreducible characters of $G$ of degree prime to $p$, if and only if $G$ has
self-normalising Sylow $p$-subgroups, and that this happens if and only if
$d_{\chi 1_G}=1$ for all $\chi\in\Irr_{p'}(G)$. Note that irreducible characters
of degree prime to $p$ lie in $p$-blocks of maximal defect, so this situation
can only happen when the principal block is the only Brauer $p$-block of
maximal defect of $G$. It is then natural to wonder what can be said when this
is not the case and one restricts attention to the principal $p$-block $B_0(G)$.
In this sense, in \cite{NT2} they conjectured that $d_{\chi 1_G}\neq 0$ for all
$\chi\in\Irr(B_0(G))$ if and only if $G$ has a normal $p$-complement. 
In this paper we consider the intersection of the two conditions in \cite{NT}
and \cite{NT2}, namely what happens if $d_{\chi 1_G}=1$ just for all $\chi\in\Irr_{p'}(B_0(G))$. Our main result is:

\begin{thmA}
 Let $G$ be a finite group and $p>3$ be a prime. Assume that all non-abelian
 simple composition factors of $G$ of order divisible by $p$ satisfy
 Property~$(*)$ below. The following are equivalent:
 \begin{enumerate}[\rm(i)]
  \item For every $\chi\in\Irr_{p'}(B_0(G))$ we have $d_{\chi 1_G}\neq 0$.
  \item For every $\chi\in\Irr_{p'}(B_0(G))$ we have $d_{\chi 1_G}=1$.
  \item For $P\in\Syl_{p}(G)$ we have $\norm G P=P\times K$ for some $K\le G$.
 \end{enumerate}
Moreover, if $G$ is $p$-solvable, this equivalence holds for every prime $p$.
\end{thmA}

Notice that, as pointed out in \cite{NT}, all irreducible characters of odd
degree of the alternating group $\fA_5$ contain the trivial character in their
2-modular reduction, and similarly, all irreducible characters of the Ree group
$^2G_2(27)$ of non-zero 3-defect contain the trivial character in their
3-modular reduction, while the respective Sylow $p$-normalisers have no normal
$p$-complement, so the equivalence in Theorem~A fails for non-solvable groups
with $p\leq 3$.

Theorem~A involves the following property that a finite non-abelian simple group
might, or might not, satisfy:

\begin{property}   \label{propB}
 Let $S$ be non-abelian simple and $p>3$ a prime dividing $|S|$. Then for
 all almost simple groups $H$ with socle $S$ and $|H:S|$ a $p$-power, there
 exists $\chi\in\Irr_{p'}(B_0(H))$ such that $d_{\chi 1_H}= 0$.
\end{property}

We prove Property~$(*)$ does hold for many simple groups: for sporadic groups,
alternating groups, and simple groups of Lie type in characteristic different
from~$p$. We also show it for some groups of Lie type in characteristic $p$.
(The general case of groups of Lie type in their defining characteristic was
also left open in \cite{NT} and \cite{NT2}.)
As the knowledge on $p$-decomposition numbers for groups of Lie type in their
own characteristic is too weak at present we refrain from making a general
conjecture and just leave it as a question as to whether Property~$(*)$ holds
for all non-abelian finite simple groups.
\medskip

\noindent
{\bf Structure of the paper.}
In Section~\ref{sec:simple} we prove Property~$(*)$ for $S$ a sporadic
simple group, an alternating group, a simple group of Lie type in characteristic
different from~$p$ as well as for some groups of Lie type in characteristic $p$.
In Section~\ref{sec:red} we show that Theorem~A holds for $p$-solvable
groups and in Theorem~\ref{thm:reduction} we reduce the general case to the
validity of Property~($*$) on composition factors, thus completing the proof
of Theorem~A.

\subsection*{Acknowledgements}
The second author would like to thank Gabriel Navarro for bringing this topic
to her attention.

\section{Almost simple groups}   \label{sec:simple}
In this section, we discuss instances of Property $(*)$ from the introduction.

\subsection{Alternating and sporadic groups}
We start out with simple groups not of Lie type.

\begin{prop}   \label{prop:spor}
 Property~$(*)$ holds for $S$ a sporadic simple group or the Tits group.
\end{prop}

\begin{proof}
Let $S$ be as in the assumption and $H\ge S$ as in Property~$(*)$. Since $p>2$
this means $H=S$. If $S$ has abelian
Sylow $p$-subgroups, by \cite[Prop.~3.2]{NT2} there exists $\chi\in\Irr(B_0(H))$
such that $d_{\chi 1_H}=0$. By one direction of Brauer's height zero conjecture
\cite{KM13}, we have $\chi\in\Irr_{p'}(B_0(H))$ as required. Else, by
inspection in \cite{GAP}, $S$ has just one $p$-block of maximal defect and then
by \cite[Lemma~3.2]{NT}, there is $\chi\in\Irr_{p'}(H)=\Irr_{p'}(B_0(H))$ with
$d_{\chi 1_H}=0$.
\end{proof}

\begin{prop}   \label{prop:alt}
 Property~$(*)$ holds for $S$ an alternating group.
\end{prop}

\begin{proof}
Let $S=\fA_n$ with $n\ge5$. As $p>2$ again we have $H=S$. Set $G:=\fS_n$.
Recall that the irreducible characters of $\fS_n$ are naturally labelled by
partitions of~$n$. If $p$ divides~$n$, then let
$\chi=\chi^{(n-2,1^2)}\in\Irr(B_0(G))$, as in \cite[Prop.~3.1]{NT2}.
Then $\chi^\circ$ is the sum of two irreducible Brauer characters of degrees
$n-2$ and $(n-2)(n-3)/2$ and hence it is of degree prime to~$p$ and does not
contain any irreducible Brauer character of degree 1. Now, any $\theta\in\Irr(S)$
under $\chi$ lies in $\Irr_{p'}(B_0(S))$ and $d_{\theta 1_S}=0$, as wanted. 

So we may assume that $p$ does not divide $n$. Let $n=a_kp^k+\cdots+a_1p+a_0$
be the $p$-adic expansion of $n$. Since $p$ does not divide $n$, we have
$a_0>0$. Suppose first that $a_0>1$. Consider $\chi=\chi^{(a_0,1^{n-a_0})}$.
By Peel's theorem  (see \cite[Thm~24.1]{Ja78}), $\chi^\circ\in\IBr(G)$. Hence,
it is enough to show that $\chi(1)\neq 1$, $\chi(1)$ is $p'$, and
$\chi\in\Irr(B_0(G))$. By the hook length formula, $\chi$ has degree
$$\chi(1)=\frac{n!}{n\cdot (a_0-1)! (n-a_0)!}=\binom{n-1}{n-a_0}.$$
Then $\chi(1)\neq 1$ since $1<a_0<n$. Moreover
$$n-1=a_kp^k+\cdots + a_1p+(a_0-1)$$ and
$$n-a_0=a_kp^k+\cdots + a_1p,$$ so by Lucas' theorem we have
$$\chi(1)=\binom{n-1}{n-a_0}\equiv \prod_{i=1}^k\binom{a_i}{a_i}\binom{a_0-1}{0}
  \equiv 1\pmod p.$$
Thus, $\chi(1)$ is not divisible by $p$. Finally since the $p$-core of
$\lambda=(a_0,1^{n-a_0})$ is $(a_0)$, we have that $\chi$ lies in the principal
$p$-block by the Nakayama conjecture, as desired. Now take $\theta\in\Irr(S)$
under $\chi$, so $\theta\in\Irr_{p'}(B_0(S))$ and $d_{\theta 1_S}=0$.
 
Finally, suppose that $a_0=1$, so $p$ divides $n-1$. In this case consider
$\chi=\chi^{(n-3,2,1)}$, so $\chi$ lies in $B_0(G)$. This is the character in
the proof of \cite[Lemma 3.1(ii)]{NT}, so we are done in this case as well.
\end{proof}

\subsection{Groups of Lie type in non-defining characteristic}
For groups of Lie type in cross characteristic we consider the following set-up.
Let $\bG$ be a simple linear algebraic group of adjoint type over an
algebraically closed field of characteristic~$r$ and $F:\bG\to\bG$ a Steinberg
map, with group of fixed points $G:=\bG^F$. It is well-know that any simple
group of Lie type can be obtained as $S=[G,G]$ for $\bG,F$ chosen suitably.
Moreover, if $\bG_\SC$ denotes a simply connected covering of $\bG$, with
corresponding Steinberg map also denoted $F$, then
$S\cong\bG_\SC^F/\bZ(\bG_\SC^F)$, if $S$ is not the Tits simple group, which was
already discussed in Proposition~\ref{prop:spor}. Let $(\bG^*,F)$ be dual to
$(\bG,F)$ and $G^*:=\bG^{*F}$.

We let $\bB\le\bG$ denote an $F$-stable Borel subgroup of $\bG$, and set
$B:=\bB^F$.

We recall that automorphisms of prime order $p\ge5$ of simple groups of Lie
type are either field automorphisms, or diagonal automorphisms for groups
of types $\PSL_n(\eps q)$, with $\eps\in\{\pm1\}$.

\begin{prop}   \label{prop:big primes}
 Property~$(*)$ holds for $S$ as above if $|B|$ is prime to $p$.
\end{prop}

\begin{proof}
By assumption $p$ divides $|S|$, hence also $|G|=|G^*|$. Let $1\ne s\in G^*$ be
a (semisimple) $p$-element in the centre of a Sylow $p$-subgroup of $G^*$, and
let $\chi\in\Irr(G)$ be the semisimple character in the Lusztig series
$\cE(G,s)$, unique since $\bG$ has connected centre, see
\cite[Def.~2.6.9]{GM20}. By the degree formula for Jordan
decomposition \cite[Cor.~2.6.6]{GM20}, $\chi(1)$ is then prime to $p$, and also
$\chi(1)>1$ as $p$ does not divide $|\bZ(G^*)|$ and so $C_{G^*}(s)<G^*$.
Furthermore, by \cite[Cor.~3.4]{Hi88}, the semisimple character $\chi$ lies in
the same $p$-block of~$G$ as the semisimple character in $\cE(G,1)$, i.e., the
trivial character, so in the principal $p$-block. Since $p$ does not
divide $|B|$, the permutation module $1_B^G$ is projective, and thus
contains the projective cover of the trivial module. But all constituents
of $1_B^G$ are unipotent, so lie in $\cE(G,1)$ (see \cite[Exmp.~3.2.6]{GM20}).
Hence, by Brauer reciprocity, $1_G^\circ$ does not occur as a constituent
of~$\chi^\circ$. Taking for $\theta$ any character of $S$ below $\chi$ we see
that $\theta\in\Irr_{p'}(B_0(S))$ and $d_{\theta 1_S}=0$.

Next observe that the order of any outer diagonal automorphism of $S$ divides
the order of $B$. Thus, $H$ is an extension of $S$ by a $p$-power order field
automorphism $\ga$. Note that any such automorphism of $S$ extends to $G$
and then also induces a dual automorphism on $G^*$, which we denote $\ga^*$
(see e.g. \cite[5.6]{Tay}). Now choose $s$ more precisely to be also
$\ga^*$-invariant (which is possible as $\ga^*$ is a $p$-element, so necessarily
has non-trivial fixed points on the centre of a $\gamma^*$-stable Sylow
$p$-subgroup of $G^*$). Then $\chi$
is also $\ga$-invariant by \cite[Prop.~7.2]{Tay}. Since the index $|G:S|$ is a
divisor of $|B|$, hence prime to~$p$, there then exists a $\ga$-invariant
$\theta\in\Irr_{p'}(B_0(S))$ below $\chi$, with $d_{\theta 1_S}=0$. Let
$\tilde\theta$ be an extension of $\theta$ to $H=S\langle\ga\rangle$ in
$B_0(H)$. Then $\tilde\theta$ is as desired.
\end{proof}

\begin{prop}   \label{prop:q-1}
 Property~$(*)$ holds for $S$ as above if $F$ is a Frobenius
 map with respect to an $\FF_q$-structure and $p$ divides $q-1$.
\end{prop}

\begin{proof}
Let $W$ be the Weyl group of $G$, that is, the $F$-fixed points of the Weyl
group of~$\bG$. As $p$ divides $q-1$, the $p$-decomposition matrix of the group
algebra of $W$ embeds into the $p$-modular decomposition matrix of $G$ (see
\cite[4.10]{Di90}). Let $\eps\in\Irr_{p'}(W)$ be the (linear) sign character.
Then $\eps^\circ\ne 1_W^\circ$ since $p>2$, whence $d_{\eps1_W}=0$. Now $\eps$
corresponds to the Steinberg character $\St$ of~$G$.
Then $d_{\St 1_G}=d_{\eps1_W}=0$, $\St$ lies in the principal $p$-block of $G$
(e.g. by \cite[Thm~A]{En00}) and its degree is a power of the
defining characteristic, so prime to~$p$. Now note that any $p$-automorphism
of~$S$ is realised inside the extension of $G$ (which induces all diagonal
automorphisms) by a generator $\ga$ of the cyclic group of $p$-power order field
automorphisms, so we may assume $H\le \tilde G:=G\langle\ga\rangle$. By
\cite[Thm~4.5.11]{GM20}, $\St$ is invariant under $\ga$. Let $\tilde\St$ be an
extension of $\St$ to $\tilde G$ in $B_0(\tilde G)$, so
$d_{\tilde\St 1_{\tilde G}}=0$. Then $\tilde\St|_H$ is irreducible, since $\St$
restricts irreducibly to $S$, hence lies in $B_0(H)$ and so is as required.
\end{proof}

\begin{thm}   \label{thm:nondef}
 Property~$(*)$ holds for $S$ of Lie type when $p$ is not the
 defining characteristic.
\end{thm}

\begin{proof}
By Proposition~\ref{prop:big primes} we may assume that $p$ divides the order
of a Borel subgroup $B$ of $G$. If $F$ is a Frobenius map with respect to an
$\FF_q$-structure and $p$ divides $q-1$, we are done by
Proposition~\ref{prop:q-1}. If $G$ is a Suzuki or Ree group and $p$ divides
$|B|$, then $p|(q^2-1)$ where $F^2$ defines an $\FF_{q^2}$-structure, and the
exactly same arguments as in the proof of Proposition~\ref{prop:q-1} apply.

So we are reduced
to the case that $F$ is a Frobenius map with respect to an $\FF_q$-structure
and $p$ divides $|B|$ but not $q-1$. Since $p$ is not the defining prime, this
implies that $G$ is a twisted group of Lie type $\tw2A_{n-1}$, $\tw2D_n$,
or $\tw2E_6$ and $p$ divides $q+1$, respectively of type $\tw3D_4$ and $p$
divides $q^2+q+1$. Let $d=2,3$ in the respective cases. Then the centraliser of
a Sylow $d$-torus of $G$ is a maximal torus, so has a unique $d$-cuspidal
unipotent character. Thus, by \cite[Thm~A]{En00} there is a unique unipotent
block of $G$ of maximal defect, the principal block, which hence contains the
Steinberg character $\St$ of $G$. By \cite[Thm~B]{Hi90} we have $d_{\St 1_G}=0$
in our case unless $G=\PGU_3(q)$. Except for that latter case, we can now argue
as in the proof of Proposition~\ref{prop:q-1} to conclude. For $G=\PGU_3(q)$
let $\chi$ be the cuspidal unipotent character of degree~$q(q-1)$, prime to~$p$.
By \cite[Thm~4.3(a)]{Ge90} it lies in $B_0(G)$ and satisfies $d_{\chi 1_G}=0$.
Again, $\chi$ restricts irreducibly to $S$ and is invariant under all
automorphisms, so we can argue as before.
\end{proof}

\subsection{Groups of Lie type in defining characteristic}
We do not see how to approach Property~$(*)$ for groups of Lie type in their
defining characteristic in general. All characters of positive defect lie in the
principal block and decomposition numbers tend to be large and little is known.
The following was shown by Navarro and Tiep:

\begin{prop}   \label{prop:Sp}
 Property~$(*)$ holds if $S=\PSp_{2n}(p^f)$, $n\geq1$, with either $p>3$ or
 $p=3$ and $n$ even.
\end{prop}

\begin{proof}
In this case the principal block of $S$ is the only $p$-block of positive
defect. Let $H$ be almost simple with $H/S$ a $p$-group. By
\cite[Cor.~9.6]{nbook} there is just one $p$-block of~$H$ covering $B_0(S)$,
necessarily the principal block $B_0(H)$. Now the irreducible character
$\chi\in\Irr_{p'}(H)$ with $d_{\chi1_H}=0$ constructed in \cite[Prop.~3.11]{NT}
lies above a character of $S$ of positive defect, hence in the principal
$p$-block of $H$ and we are done.
\end{proof}

Observe that this does not extend to $p=3$ and $n$ odd: The group
$H=\PSp_2(3^3).3$ has no irreducible character $\chi\in\Irr_{3'}(H)$ with $d_{\chi 1_H}=0$.
\smallskip

The paper \cite{NT2} also contains results for special linear and unitary
groups but these are not applicable here as the considered characters are not
of $p'$-degree. Nevertheless, we can follow their general approach.

For $G=\SL_n(q)$ we let $\tau_j$, $j=1,\ldots,q-2$ denote the non-unipotent
\emph{Weil characters} of degree $(q^n-1)/(q-1)$, ordered such that $\tau_j$ is
trivial on the centre $Z(\SL_n(q))$
of order $z:=\gcd(n,q-1)$ if and only if $z|j$.

\begin{prop}   \label{prop:SL}
 Let $S=\PSL_n(q)$ with $q=p^f$, $p\ne2$ and $n\ge3$.
 \begin{enumerate}[\rm(a)]
  \item If either $\gcd(p-1,(q-1)/\gcd(n,q-1))>1$ or $2^f<(q-1)/z-1$ then there
   is $\chi\in\Irr_{p'}(B_0(S))$ such that $d_{\chi 1_S}= 0$.
  \item Write $f=p^af'$ with $\gcd(p,f')=1$ and set $q':=p^{f'}$. If $a=0$ or
   $2^f<(q'-1)/(q'-1,n)-1$ then Property~$(*)$ holds for $S=\PSL_n(q)$.
 \end{enumerate}
\end{prop}

\begin{proof}
Let $G:=\SL_n(q)$ and set $z=\gcd(n,q-1)$. We are interested in the characters
$\tau_j$ of $G$ that are trivial on $Z(G)$, that is, for which $j=zj'$ for some
integer  $1\le j'\le (q-1-z)/z$. By \cite[Thm~1.11]{ZS} the decomposition number
$d_{\tau_j1_G}$ equals the number of solutions $x_s\in\{0,1\}$ of the congruence
$$n(p-1)\sum_{s=0}^{f-1}x_sp^s\equiv j\pmod{(q-1)}.$$
Dividing by $z$, we need to count solutions to
$$n/z(p-1)\sum_{s=0}^{f-1}x_sp^s\equiv j'\pmod{(q-1)/z}.$$
If there is a prime $\ell$ dividing $p-1$ and $(q-1)/z$, then reducing
modulo $\ell$ we see there is no solution for $j'=1$. Also, the left hand side
can take at most $2^f$ distinct values. Since there are $(q-1)/z-1$ admissible
values for $j'$, there is $j'$ with no solutions whenever $2^f<(q-1)/z-1$.
Thus, under either of our assumptions we find $j'$ with
$\tau_j=\tau_{zj'}\in\Irr_{p'}(G)$ with $d_{\tau_j1_G}=0$. Since $G$ has a
single $p$-block of positive defect, $\tau_j$ lies in the principal block.
Furthermore, by construction $Z(\SL_n(q))$ lies in the kernel of $\tau_j$ and
hence $\tau_j$ deflates to a character of $S=\PSL_n(q)$. Thus we get~(a).
\par
(b) Since $p>2$ does not divide $q-1$, the $p$-power order automorphisms of $S$
are field automorphisms, of order dividing $p^a$ where $f=p^af'$ is as in the
statement. If $a=0$ Property~$(*)$ follows from~(a) since necessarily $H=S$.
For $a>0$ let $\ga$ be a field automorphism of $G$ (and hence of $S$) of order
$p^a$. There are exactly $q'-2$ Weil characters of $G$ invariant under $\ga$,
which hence extend to $G\langle\ga\rangle$. Of these, $(q'-1)/\gcd(q'-1,n)-1$
are trivial on $Z(G)$, so define characters in $\Irr_{p'}(S)$ invariant
under $\ga$. By the argument above, if this number is bigger than $2^f$ then
there exists such a character $\chi$ with $d_{\chi 1_S}=0$. Hence any character
of $H$ in $B_0(H)$ lying above it verifies Property~$(*)$.
\end{proof}

Note that the case $n=2$ is contained is addressed in Proposition~\ref{prop:Sp}.
Observe that the condition in Proposition~\ref{prop:SL}(a) is satisfied if
$\gcd(n,q-1)=1$, for example. It also holds when $p>n+1$ (since $p-1$ always
divides $q-1$), or if $q>n(2^f+1)$. Thus, Proposition~\ref{prop:SL} extends
and complements \cite[Prop.~3.3(ii)]{NT2}.

\begin{cor}
 Let $S=\PSL_n(q)$ with $q=p^f$, $p\ne2$ and $3\le n\le9$. Then there is
 $\chi\in\Irr_{p'}(B_0(S))$ with $d_{\chi 1_S}= 0$ unless possibly $S$ is one of
   $$\PSL_4(5),\PSL_6(3),\PSL_6(7),\PSL_8(3),\PSL_8(9),\PSL_8(5),\PSL_8(25).$$
\end{cor}

\begin{proof}
By the conditions in Proposition~\ref{prop:SL}(a) and directly checking with
the formula of Zalesski--Suprunenko if these fail ones finds that the only cases
with $n\le9$ not covered are $\PSL_4(3)$ and those listed in the statement.
The decomposition matrix of $\PSL_4(3)$ is in \cite{GAP} from which the claim
can be verified.
\end{proof}

For $G=\SU_n(q)$ we let $\tau_j$, $j=1,\ldots,q$ denote the non-unipotent Weil
characters, constructed by Seitz \cite{Se75}, of degree $(q^n-(-1)^n)/(q+1)$,
again ordered such that $\tau_j$ is trivial on the centre $Z(\SU_n(q))$ of
order $z:=\gcd(n,q+1)$ if and only if $z|j$.

\begin{prop}   \label{prop:SU}
 Let $S=\PSU_n(q)$ with $q=p^f$, $p\ne2$ and $n\ge3$.
 \begin{enumerate}[\rm(a)]
  \item If $2^f<(q+1)/z-1$ then there is $\chi\in\Irr_{p'}(B_0(S))$ such that
   $d_{\chi 1_S}= 0$.
  \item Write $f=p^af'$ with $\gcd(p,f')=1$ and set $q':=p^{f'}$. If $a=0$ or
   $2^f<(q'+1)/(q'+1,n)-1$ then Property~$(*)$ holds for $S=\PSU_n(q)$.
 \end{enumerate}
\end{prop}

\begin{proof}
The argument is very similar to the one for the special linear groups. Let
$G=\SU_n(q)$ and set $z=\gcd(n,q+1)=|Z(G)|$. Again, we consider characters
$\tau_j$ trivial on $Z(G)$, that is, for which $j=zj'$ for some integer
$1\le j'\le (q+1-z)/z$.
By \cite[Main Thm]{Za} the decomposition number $d_{\tau_j1_G}$ equals the
number of solutions $x_s\in\{0,1\}$ of the congruence
$$n\big((p-1)\sum_{s=0}^{f-1}x_sp^s-1\big)\equiv j\pmod{(q+1)}.$$
(In fact, \cite{Za} has an additional summand of $\frac{q+1}{2}$ on the right
hand side, but this disappears here due to a different numbering of the
$\tau_j$, see \cite[p.~612]{NT2}; in any case,
this difference will not matter for our argument here.) Since the left-hand side
can take at most $2^f$ distinct values, while there are $(q+1)/z-1$ admissible
values for $j'$ the assertion in (a) follows. For (b) we can argue exactly as
in the proof of Proposition~\ref{prop:SL}.
\end{proof}

As in \cite{NT,NT2} we have no general results for orthogonal or exceptional
type groups in their defining characteristic.

\section{The reduction}   \label{sec:red}
In this section we prove Theorem A. We will need the following results, which we collect here for the reader's convenience.

\begin{lem}   \label{lem:murai}
 Let $N\triangleleft G$ and let $\theta\in\Irr_{p'}(B_0(N))$. Suppose that
 $\theta$ extends to $PN$, where $P\in\Syl_p(G)$. Then there exists
 $\chi\in\Irr_{p'}(B_0(G))$ satisfying $[\theta^G,\chi]\neq 0$. 
\end{lem}

\begin{proof}
This is \cite[Lemma 4.3]{Mur}.
\end{proof}

The following argument is inside the proof of \cite[Thm 2.6]{NT}.

\begin{lem}   \label{lem:decquotientpgroup}
 Let $G$ be a finite group and suppose that $d_{\chi 1_G}\neq 0$ for every
 $\chi\in\Irr_{p'}(B_0(G))$. Let $M\triangleleft G$ and let  $P\in \Syl_p(G)$.
 Then $d_{\psi 1_{MP}}\neq 0$ for every  $\psi\in\Irr_{p'}(B_0(MP))$.
\end{lem}

\begin{proof}
Since $MP/M$ is a $p$-group, we have that $\psi_M=\tau\in\Irr_{p'}(B_0(M))$.
By Lemma~\ref{lem:murai} there exists $\chi\in\Irr_{p'}(B_0(G))$ lying
over~$\tau$. By hypothesis, $d_{\chi 1_G}\neq 0$, and then $\chi^\circ_M$
contains $1_M$, so $d_{\tau 1_M}\neq 0$. Since $MP/M$ is a $p$-group, we have
that $(MP)^\circ=M^\circ$ and then $d_{\psi 1_{MP}}\neq 0$, as wanted.
\end{proof}

We next prove Theorem A for $p$-solvable groups.

\begin{thm}   \label{thm:psolvable}
 Let $G$ be a $p$-solvable group. Then the following are equivalent:
 \begin{enumerate}[\rm(i)]
  \item For every $\chi\in\Irr_{p'}(B_0(G))$ we have $d_{\chi 1_G}\neq 0$.
  \item For every $\chi\in\Irr_{p'}(B_0(G))$ we have $d_{\chi 1_G}=1$.
  \item For $P\in\Syl_{p}(G)$ we have $\norm G P=P\times K$ for some $K\le G$.
 \end{enumerate}
\end{thm}

\begin{proof}
We first prove (iii) implies~(ii). By \cite[Thm~3.2]{NTV} we have
$K\subseteq\bO_{p'}(G)=:X$. Now, let $\overline{G}=G/X$ and $\overline{P}=PX/X$,
then (iii) implies that $\norm {\overline{G}} {\overline{P}}\cong\overline{P}$.
Since $\Irr_{p'}(B_0(G))=\Irr_{p'}(B_0(G/X))$ we know by \cite[Thm~B]{NT}
that (i) and (ii) hold.
\medskip

Since (ii) implies (i) trivially, we just need to show that (i) implies (iii).
We proceed by induction
on $|G|$. Let $N=\bO_{p'}(G)$ and
use the bar notation. Since $\Irr_{p'}(B_0(G))=\Irr_{p'}(B_0(\bar{G}))$, if
$N>1$, we have by induction that $\norm {\bar G}{\bar P}=\bar{P}\times\bar{K}$.
By \cite[Thm~3.2]{NTV} we have that $\bar{K}\subseteq \bO_{p'}(\bar G)=1$ so
$\bar K=1$ and hence $\norm {\bar G}{\bar P}=\bar{P}$. This implies that
$\norm G P=P\times\cent N P$ and we are done. So we may assume that $N=1$. But
in this case the principal $p$-block is the only $p$-block of $G$, and we are
done by \cite[Thm~B]{NT}.
\end{proof}

We finally prove Theorem A.

\begin{thm}   \label{thm:reduction}
 Let $G$ be a finite group. Assume that $p>3$ and all non-abelian composition
 factors of $G$ of order divisible by $p$ satisfy Property~$(*)$. Then the
 following are equivalent:
 \begin{enumerate}[\rm(i)]
  \item For every $\chi\in\Irr_{p'}(B_0(G))$ we have $d_{\chi 1_G}\neq 0$.
  \item For every $\chi\in\Irr_{p'}(B_0(G))$ we have $d_{\chi 1_G}=1$.
  \item For $P\in\Syl_{p}(G)$ we have $\norm G P=P\times K$ for some $K\le G$.
 \end{enumerate} 
\end{thm}

\begin{proof}
As in Theorem \ref{thm:psolvable}, we just need to show that (i) implies (iii).
We work to prove this implication in a series of steps, proceeding by induction
on $|G|$.
\medskip

\textit{Step 1. We may assume $\bO_{p'}(G)=1$}. Indeed, let $N=\bO_{p'}(G)$ and
use the bar notation. Since $\Irr_{p'}(B_0(G))=\Irr_{p'}(B_0(\bar{G}))$, if
$N>1$, we have by induction that $\norm {\bar G}{\bar P}=\bar{P}\times\bar{K}$.
By \cite[Thm~3.2]{NTV} we have that $\bar{K}\subseteq \bO_{p'}(\bar G)=1$ so
$\bar K=1$ and hence $\norm {\bar G}{\bar P}=\bar{P}$. This implies that
$\norm G P=P\times\cent N P$ and we are done.
\medskip

\textit{Step 2. If $1<M\lhd G$  then $M$ is not $p$-solvable. In particular,
if $N$ is a minimal normal subgroup of $G$, then $N$ is semisimple of order
divisible by $p$}.
Let $M\lhd G$ be the largest $p$-solvable normal subgroup of~$G$. We claim that
$M=1$. Let $\bar{G}=G/M$ and use the bar notation. Suppose $M>1$.
Since $\Irr_{p'}(B_0(\bar{G}))\subseteq\Irr_{p'}(B_0(G))$ we have by
induction that $\norm {\bar{G}}{\bar{P}}=\bar{P}\times\bar{K}$.
By \cite[Thm~3.2]{NT} this implies that $\bar{K}\subseteq\bO_{p'}(\bar{G})$.
Since $M$ is the largest normal $p$-solvable subgroup of $G$,
$\bO_{p'}(\bar{G})$ is trivial and hence $\norm {\bar{G}}{\bar{P}}=\bar{P}$.
By \cite[Thm~1.1]{GMN} this forces $\bar{G}$ to be solvable. Hence $G$ is
$p$-solvable and we are done by Theorem \ref{thm:psolvable}.
\medskip

From now on, let $N=S_1\times \cdots\times S_t$ be a minimal normal subgroup of $G$, where $S_i\cong S$ is a nonabelian simple group of order divisible by $p$.
\medskip

\textit{Step 3. Let $X/N=\bO_{p'}(G/N)$. Then
$G=XP$.}
Since $\Irr_{p'}(B_0(G/N))\subseteq\Irr_{p'}(B_0(G))$, by induction we have
$\norm G P N/N=\norm {G/N}{PN/N}=PN/N\times K/N$ for some $K\le G$. Then
$K/N\subseteq X/N$ by \cite[Thm~3.2]{NTV}. Write $H=XN\norm G P=XPNK=XP$.
By Lemma~\ref{lem:decquotientpgroup} applied to $X$ in place of $M$, we obtain
that $H$ satisfies~(i). Since the non-abelian composition factors of $H$ are composition factors of $G$, if $H<G$, by induction we have $\norm G P=\norm H P=P\times L$ for some $L$ and we are done. Hence we may assume that $H=G$.
\medskip

\textit{Step 4.  Let $H=\bigcap\norm {G} {S_i}$ and let $X$ be as in Step 3. Then $G=(H\cap X)P$.}
Write $Y=(H\cap X)P$. We claim first that $\bO_{p'}(Y)=1$. Indeed,
$\bO_{p'}(Y)\cap (H\cap X)\subseteq\bO_{p'}(H\cap X)=1$ by Step 1. Now,
$\bO_{p'}(Y)\cong (H\cap X)\bO_{p'}(Y)/(H\cap X)$ is a $p$-group, so
$\bO_{p'}(Y)=1$ as wanted. By Lemma \ref{lem:decquotientpgroup} applied to
$H\cap X$ in place of $M$, we have that $Y$ satisfies (i). Since the non-abelian composition factors of $Y$ are composition factors of $G$, if $Y<G$, by
induction this gives $\norm Y P=P\times K$ for some $K$. But then
$K\subseteq\bO_{p'}(Y)$ by \cite[Thm~3.2]{NTV}, so $K=1$. This means that
$\norm Y P=P$ and then by \cite[Thm~1.1]{GMN} the group $Y$ is solvable.
But then $N$ is solvable, a contradiction.   Hence $Y=G$ as wanted.
\medskip

\textit{Step 5. Final step.}
Let $H=\bigcap\norm {G} {S_i}$ and $X$ be as before. Since $G=(H\cap X)P$ and $H$ acts trivially on $\{S_1,\ldots, S_t\}$, $P$ must act transitively on the set $\{S_1,\ldots, S_t\}$. Write $S=S_1$ and, for $i=2,\ldots,t$, write $S_i=S^{x_i}$ with $x_i\in P$. We proceed now as in the proof of \cite[Thm~2.6]{NT}.
Let $R=\norm P S$. If $SR=G$, then $S=N$ and $\cent G S$ is a normal $p$-subgroup of $G$. By Step~2, $\cent G S=1$ and hence $G$ is almost simple with socle $S$ and $|G:S|$ a power of $p$. Since $S$ satisfies Property~$(*)$ by
assumption, we have a contradiction. Hence we may assume that $SR<G$.

Let $Q=P\cap N$ and let $R_1=R\cap S=Q\cap S=P\cap S\in\Syl_p(S)$. Let $\gamma\in\Irr_{p'}(B_0(SR))$ and notice that $\gamma_S=\psi\in\Irr_{p'}(B_0(S))$ since $SR/S$ is a $p$-group. For $i=2,\ldots,t$, let $\psi_i=\psi^{x_i}\in\Irr_{p'}(B_0(S_i))$ and let $\eta=\psi\times\psi_2\times\cdots\times \psi_t\in\Irr_{p'}(B_0(N))$, which is $P$-invariant by \cite[Lemma~4.1(ii)]{NTT}. Then $\eta$ extends to $PN$. By Lemma~\ref{lem:murai} and hypothesis we have $d_{\eta 1_N}\neq 0$. By \cite[Lemma~2.3]{NT} we have that $d_{\psi 1_S}\neq 0$ and then, since $SR/S$ is a $p$-group, we conclude that $d_{\gamma 1_{SR}}\neq 0$. Since $SR<G$ and $S$ is a composition factor of $G$, this implies $\norm {SR} R=R\times K$ for some $K$. Then $K\subseteq\bO_{p'}(SR)$. But $\bO_{p'}(SR)=1$ (as before, $\bO_{p'}(SR)\cap S=1$, so $\bO_{p'}(SR)\cong S\bO_{p'}(SR)/S$ is a $p$-group, and hence trivial), so $K=1$ and then $\norm {SR} R=R$. Now by \cite[Thm~1.1]{GMN}, $SR$ is solvable, and hence $S$ is solvable, a final contradiction.
\end{proof}

We remark that, as happens in \cite{NT}, (iii)$\Rightarrow$(ii)$\Rightarrow$(i)
is always true if $p>3$, as is shown in the first lines of the proof of
Theorem~\ref{thm:psolvable} (notice that that part does not require
$p$-solvability, since \cite[Thm~3.2]{NTV} holds for odd primes).


\end{document}